\documentclass[12pt]{amsart}
\usepackage[top=1in, bottom=1in, left=1in, right=1in]{geometry}
\usepackage{amsfonts}
\usepackage{amsmath}
\usepackage{comment}
\usepackage{amssymb}
\usepackage{bbm}
\usepackage{comment}
\usepackage{mathrsfs}
\numberwithin{equation}{section}
\usepackage{times}

\usepackage{siunitx}
\usepackage[usenames,dvipsnames]{color}
\usepackage{comment}
\usepackage{xcolor}
\usepackage{mathtools}
\usepackage{bm}
\usepackage{esvect}
\usepackage{hyperref}
\hypersetup{
    colorlinks = true,
linkcolor={black},
urlcolor={blue},
citecolor={blue},    
urlcolor = {blue},
citebordercolor = {0.33 .58 0.33},
 linkbordercolor = {0.99 .28 0.23},
 breaklinks=true}

\newcommand{\kommentar}[1]{}
\newcommand{\C}{\mathbb{C}}

\newcommand{\R}{\mathbb{R}}

\newcommand{\N}{\mathbb{N}}

\newtheorem{thm}{Theorem}[section]

\newtheorem{coro}[thm]{Corollary}

\newtheorem{rem}[thm]{Remark}

\newcommand{\dd}{\;\mathrm{d}}

\theoremstyle{remark}

\usepackage{geometry}
\title{An Equidistribution Result for Differences Associated to Square Pyramidal Numbers II}

\author{Anji Dong, Katerina Saettone, Kendra Song and Alexandru Zaharescu}
\address{
Anji Dong: Department of Mathematics,
University of Illinois Urbana-Champaign,
Altgeld Hall, 1409 W. Green Street,
Urbana, IL, 61801, USA}
\email{anjid2@illinois.edu}

\address{
Katerina Saettone: Department of Mathematics,
University of Illinois Urbana-Champaign,
Altgeld Hall, 1409 W. Green Street,
Urbana, IL, 61801, USA}
\email{kas18@illinois.edu}

\address{
Kendra Song: Department of Mathematics,
University of Illinois Urbana-Champaign,
Altgeld Hall, 1409 W. Green Street,
Urbana, IL, 61801, USA}
\email{kendras4@illinois.edu}

\address{
Alexandru Zaharescu: Department of Mathematics,
University of Illinois Urbana-Champaign,
Altgeld Hall, 1409 W. Green Street,
Urbana, IL, 61801, USA and Simion Stoilow Institute of Mathematics of the Romanian Academy, 
P. O. Box 1-764, RO-014700 Bucharest, Romania}
\email{zaharesc@illinois.edu}

\begin{document}
\keywords{Cannonball problem, exponential sums, discrepancy, equidistribution, Dirichlet series, arithmetic progressions. 
}
\subjclass{Primary: 11K06. Secondary: 11K38, 11B99, 11L07, 11M99}
\begin{abstract}
  This paper presents some new results concerned with uniform distribution properties associated with the sequence $(a_n)_{n\in\N}$, which is defined as the distance from the $n$-th square pyramidal number to the closest square. We also extend the results to arithmetic progressions.
\end{abstract}
\maketitle

\section{Introduction}\label{introduction}
The Cannonball problem, officially proposed by Lucas \cite{LucasQuestions} and solved by Watson \cite{WatsonCannonball}, asks which integers are both a square and a square pyramidal number. For more works related to the Cannonball problem, the reader is referred to Conway and Sloan \cite{conway1982}, \cite{conway1999}, Laub \cite{Laub1990}, Beeckmans \cite{beeckmans1994}, and Bennett \cite{bennett2002}.

In connection with the Cannonball problem, Paolo Xausa \cite{paolo} considered the sequence \hyperlink{https://oeis.org/A351830}{A351830} defined as follows:
\begin{align}\label{a_n first definition}
    a_{n} = \big\lvert P_{n} - y^{2}_{n} \big\rvert,
\end{align}
where $P_{n}$ is the $n$-th square pyramidal number defined as
\begin{align}
   P_n &= \sum_{i = 1}^n i^2= \frac{2n^3+3n^2+n}{6},\label{defn: square pyramidal number} 
\end{align} 
and $y^{2}_{n}$ is the closest square to $P_{n}$. Hence, $a_n$ is $0$ if and only if $P_n$ is a solution to the Cannonball problem, and this happens exactly when $n=0$, $1$, and $24$. Xausa also computed the first ten thousand elements of the sequence.

In \cite{equidistribution_paper_1}, the authors supplied an asymptotic formula for the average value of the sequence defined in \eqref{a_n first definition} and \eqref{defn: square pyramidal number}. 
To be precise, the authors proved the following result \cite[Thm\ 1.1]{equidistribution_paper_1}\hspace{0.1cm}:

\bigskip

\textit{
For any $x\geq 1$, define
\begin{align}\label{A(x) first definition}
A(x) := \frac{1}{x}\sum_{1 \leq n \leq x} a_{n}.
\end{align}
Then, the function $A(x)$ satisfies
\begin{align}\label{A(x) asymptotic}
    A(x) = \frac{1}{5\sqrt{3}}x^{3/2}+O(x^{17/12}).
\end{align}
}
\bigskip

The above result was also generalized in \cite{equidistribution_paper_1} to any $k$-th moment with $k\in\N$. In the present paper, we continue this investigation. A natural extension of the above theorem is to consider asymptotics similar to \eqref{A(x) first definition} and \eqref{A(x) asymptotic}, where $n$ runs over a given arithmetic progression modulo a positive integer $q$.

\begin{thm}\label{thm: A(x) and M(x) in AP} For any $b,q\in\N$ and any $x\geq 1$, define 
\begin{align}\label{A(x) in AP}
A(b,q,x) := \frac{1}{x}\sum_{\substack{1 \leq n \leq x\\n\equiv b\bmod q}} a_{n},
\end{align}
where $a_n$ is as in \eqref{a_n first definition} and \eqref{defn: square pyramidal number}. Then, $A(b,q,x)$ satisfies
\begin{align}\label{A(x) in AP asymptotic}
    A(b,q,x) = \frac{1}{5q\sqrt{3}}x^{3/2}+O\left(\frac{x^{17/12}}{q^{2/3}}\right).
\end{align}
\end{thm}
\noindent\textbf{Remark.} Notice that the main term for $A(b,q,x)$ dominates the error term as long as $q\ll x^{1/4}$.

\begin{coro}\label{thm: asymptotic for dirichlet series of chi(n)a(n)}
    Let $q\in\N$, $q\geq 2$, and $a_n$ be as in \eqref{a_n first definition} and \eqref{defn: square pyramidal number}. Let $\chi$ be a Dirichlet character modulo $q$. Then, for any $x\geq 1$, 
\[
\sum_{1\leq n\leq x}a_n\chi(n)  =
\begin{cases}
\frac{\varphi(q)}{5q\sqrt{3}}x^{5/2}+O\left(\frac{\varphi(q)x^{29/12}}{q^{2/3}}\right), & \text{if } \chi = \chi_0, \\
O\left(\frac{\varphi(q)x^{29/12}}{q^{2/3}}\right), & \text{if } \chi \neq \chi_0.
\end{cases}
\]

\end{coro}

A very desirable feature of the above results is the power savings in the error terms. One can exploit this feature to derive asymptotic formulas for some other related intricate sequences. As an example of such results, we prove the following: 
\begin{thm}\label{thm: asymptotic for bn}
    For any $n\in\N$, define $b_n = \sum_{d\mid n} a_d$, where $a_d$ is given by \eqref{a_n first definition} and \eqref{defn: square pyramidal number}. Then for any $x\geq 1$, we have 
    \begin{align}
    \sum_{1\leq n\leq x}b_n \left(1-\frac{n}{x}\right)&=\frac{2\zeta(5/2)x^{5/2}}{35\sqrt{3}}+O_\delta(x^{29/12+\delta}),
    \end{align}
    for any $\delta>0$, where $\zeta(5/2)$ is the value of the Riemann zeta function at $5/2$.
\end{thm}


\subsection*{Structure of the Paper}
The paper is organized as follows. We begin by introducing some standard notation in Section \ref{sec: general notations}. In Section \ref{sec: proof of thm: A(x) and M(x) in AP}, we prove Theorem \ref{thm: A(x) and M(x) in AP} using an argument similar to the proof in \cite[ Theorem 1.2]{equidistribution_paper_1}, which employs the Erd\"{o}s-Tur\'{a}n inequality and an equidistribution theorem of Kuipers and Niederreiter \cite{UniformDistributionofSequences}. Additionally, we establish an asymptotic formula for the twisted sum associated with the sequence $a_{n}$ in Corollary \ref{thm: asymptotic for dirichlet series of chi(n)a(n)}. Lastly, in Section \ref{sec: proof of thm: asymptotic for bn} we bring into play the associated Dirichlet series in order to finish the proof of Theorem \ref{thm: asymptotic for bn}.

\section{General Notations}\label{sec: general notations}
We employ some standard notation that will be used throughout the article.
\begin{itemize} 
\item The expressions $f(X)=O(g(X))$, $f(X) \ll g(X)$, and $g(X) \gg f(X)$ are equivalent to the statement that $|f(X)| \leqslant C|g(X)|$ for all sufficiently large $X$, where $C>0$ is an absolute constant. A subscript of the form $\ll_{\alpha}$ means the implied constant may depend on the parameter $\alpha$. Dependence on several parameters is indicated similarly, as in $\ll_{\alpha, \lambda}$.
\item Given a set $S$, the notation $\#S$ stands for the cardinality of $S$. 
\item For any $x\in\R$, the notation $\lfloor x\rfloor$ denotes the floor of $x$, which is the largest integer smaller than $x$. 
\item For any real $x\in\R$, $\{x\}$ denotes the fractional part of $x$, that is, $x-\lfloor x\rfloor$.
\item The notation $e(x)$ stands for $\exp(2\pi i x)$.
\end{itemize}

\section{Proof of Theorem \ref{thm: A(x) and M(x) in AP}}\label{sec: proof of thm: A(x) and M(x) in AP}
\begin{proof}
Let $M(b,q,x)=x\cdot A(b,q,x)$. Moreover, let $L\in\N$, which grows with $x$ and will be optimized later. Following the proof in \cite[Theorem 1.2]{equidistribution_paper_1}, we obtain 
    \begin{align}
         M(b,q,x) = \bigg( \frac{2}{\sqrt{3}} \bigg)\sum_{j=1}^{L/2} \bigg(\frac{j}{L}\bigg)\sum_{\substack{1\leq n\leq x\\n\equiv b\bmod q\\\frac{j-1}{L}<|\sqrt{P_n}-y_n|\leq\frac{j}{L}}} n^{\frac{3}{2}} + O\bigl( x^{\frac{3}{2}} \bigr).\label{A(b,q,x) expression}
    \end{align}
 Let $M\leq x$ be in $\N$, and define  \begin{align}\label{S j}
S_j := \sum_{\substack{1\leq n\leq x\\n\equiv b\bmod q\\\frac{j-1}{L}<|\sqrt{P_n}-y_n|\leq\frac{j}{L}}} n^{\frac{3}{2}}  = \sum_{0\leq \ell \leq M-1}S_{j,\ell, M},
\end{align}
where 
\begin{align}\label{S(j,ell)}
    S_{j,\ell, M} := \sum_{\substack{\frac{\ell x}{M} \leq n < \frac{(\ell+1)x}{M}\\n\equiv b\bmod q\\\frac{j-1}{L} \leq |\sqrt{P_n}-y_n| < \frac{j}{L}}}n^{\frac{3}{2}}.
\end{align}
By applying Equation (4.13) in \cite{equidistribution_paper_1}, we have
\begin{align}
    S_{j,\ell, M} 
    &= \sum_{\substack{\frac{\ell x}{M} \leq n < \frac{(\ell +1)x}{M}\\n\equiv b\bmod q\\\frac{j-1}{L} < \{\sqrt{P_n}\} \leq \frac{j}{L}}} \bigg(\frac{\ell x}{M}\bigg)^{\frac{3}{2}} + \sum_{\substack{\frac{\ell x}{M} \leq n < \frac{(\ell +1)x}{M}\\n\equiv b\bmod q\\1-\frac{j}{L} < \{\sqrt{P_n}\} \leq 1-\frac{j-1}{L}}} \bigg(\frac{\ell x}{M}\bigg)^{\frac{3}{2}} \notag\\
    &\quad+O\bigg(\sum_{\substack{\frac{\ell x}{M} \leq n < \frac{(\ell +1)x}{M}\\n\equiv b\bmod q\\ \frac{1}{2}-x^{-\frac{3}{4}} < \{\sqrt{P_n}\} \leq \frac{1}{2}+x^{-\frac{3}{4}}}} \bigg(\frac{\ell x}{M}\bigg)^{\frac{3}{2}} +\sum_{\substack{\frac{\ell x}{M} \leq n < \frac{(\ell +1)x}{M}\\n\equiv b\bmod q\\\frac{j-1}{L} <|\sqrt{P_n}-y_n| \leq \frac{j}{L}}}\frac{\ell^{\frac{1}{2}}x^{\frac{3}{2}}}{M^{\frac{3}{2}}}\bigg).\label{split Sjl sum}
\end{align}
Next, we concentrate on the first sum in \eqref{split Sjl sum}. We remark that
\begin{align*}
    \sum_{\substack{\frac{\ell x}{M} \leq n < \frac{(\ell +1)x}{M}\\n\equiv b\bmod q\\\frac{j-1}{L} < \{\sqrt{P_n}\} \leq \frac{j}{L}}} \bigg(\frac{\ell x}{M}\bigg)^{\frac{3}{2}} &= \bigg(\frac{\ell x}{M}\bigg)^{\frac{3}{2}} \# \bigg\{n \in \bigg[\frac{\ell x}{M}, \frac{(\ell + 1) x}{M}\bigg] \colon \{\sqrt{P_n}\} \in \bigg(\frac{j-1}{L}, \frac{j}{L}\bigg),n\equiv b\bmod q \bigg\}\\
    &= \bigg(\frac{\ell x}{M}\bigg)^{\frac{3}{2}} \bigg(\frac{x}{qLM}+O(D(\mathcal{U}_{\ell,M}))\bigg),
\end{align*}
where
\begin{align}
    \mathcal{U}_{\ell,M} = \bigg\{ \{\sqrt{P_{n_i}}\} \colon n_i = \bigg\lfloor \frac{\ell x}{M}\bigg\rfloor +v+ iq, \ 1 \leq i \leq \bigg\lfloor\frac{\lfloor\frac{x}{M}\rfloor-v}{q}\bigg\rfloor\bigg\}, \label{Family U definition}
\end{align}
and $0\leq v<q$ is the smallest integer such that $\lfloor \frac{\ell x}{M}\rfloor +v\equiv b\bmod q$ and  $D(\mathcal{U}_{\ell,M})$ represents the discrepancy of the family $\mathcal{U}_{\ell,M}$. Applying the same arguments to the second sum and error term in \eqref{split Sjl sum}, we deduce that
\begin{align}
     S_{j,\ell, M} 
    &= \bigg(\frac{\ell x}{M}\bigg)^{\frac{3}{2}} \cdot \frac{2x}{qLM}+O\bigg(\bigg(\frac{\ell x}{M}\bigg)^{\frac{3}{2}}D(\mathcal{U}_{\ell,M})+\frac{\ell^{\frac{3}{2}}x^{\frac{7}{4}}}{qM^{\frac{5}{2}}}+\frac{\ell^{\frac{1}{2}}x^{\frac{5}{2}}}{qLM^{\frac{5}{2}}}\bigg).\label{S(j,l,M) second version}
\end{align}
To bound $D(\mathcal{U}_{\ell,M})$, we apply the Erd\"{o}s-Tur\'{a}n inequality. Then, for any $K \geq 1$,
\begin{align}
    D(\mathcal{U}_{\ell,M}) \leq \bigg\lfloor\frac{\lfloor\frac{x}{M}\rfloor-v}{q}\bigg\rfloor\frac{1}{K+1} + 3\sum_{m=1}^K\frac{1}{m}\bigg|\sum_{1 \leq i \leq \left\lfloor\frac{\left\lfloor\frac{x}{M}\right\rfloor-v}{q}\right\rfloor} e(m\sqrt{P_{n_i}})\bigg|. \label{erdos turan application}
\end{align}
We now proceed to obtain an upper bound for the exponential sum in \eqref{erdos turan application}. By Kuipers and Niederreiter's result in \cite[Theorem 2.7]{UniformDistributionofSequences}, taking $h(x) := \sqrt{P_{x}}$, we have
\begin{align}
    \sum_{1 \leq i \leq \left\lfloor\frac{\left\lfloor\frac{x}{M}\right\rfloor-v}{q}\right\rfloor} &e(m\sqrt{P_{n_i}}) \notag\\
    &\leq \bigg(m\bigg|h'\bigg(\bigg\lfloor \frac{\ell x}{M}\bigg\rfloor+v+\left\lfloor\frac{\left\lfloor\frac{x}{M}\right\rfloor-v}{q}\right\rfloor q \bigg)-h'\bigg(\bigg\lfloor \frac{\ell x}{M}\bigg\rfloor+v\bigg)\bigg| + 2\bigg)\notag\\
    &\quad\times\bigg(4\bigg(mh''\bigg(\bigg\lfloor \frac{\ell x}{M}\bigg\rfloor+v+\left\lfloor\frac{\left\lfloor\frac{x}{M}\right\rfloor-v}{q}\right\rfloor q\bigg)^{-1/2}+3\bigg)\notag\\
    &\leq \bigg(m\bigg|h'\bigg(\bigg\lfloor \frac{(\ell +1)x}{M} \bigg\rfloor \bigg)-h'\bigg(\bigg\lfloor \frac{\ell x}{M}\bigg\rfloor\bigg)\bigg| + 2\bigg)\bigg(4\bigg(mh''\bigg(\bigg\lfloor \frac{(\ell +1)x}{M} \bigg\rfloor \bigg)\bigg)^{-1/2}+3\bigg),\label{3.7}
\end{align}
where the last inequality follows from the fact that $h'(x)$ is monotonically increasing and $h''(x)$ is monotonically decreasing for $x\geq 1$. Note that the bound in \eqref{3.7} is exactly the same as in \cite{equidistribution_paper_1}, and thus we obtain 
\begin{align}
     D(\mathcal{U}_{\ell,M})
    &\ll \frac{x}{qKM} + \frac{K^{\frac{1}{2}}x^{\frac{3}{4}}}{\ell^{\frac{1}{4}}M^{\frac{3}{4}}}+\frac{K\sqrt{x}}{\sqrt{M\ell}}+\frac{\ell^{\frac{1}{4}}x^{\frac{1}{4}}}{K^{\frac{1}{2}}M^{\frac{1}{4}}}.\label{discrepancy bound 2}
\end{align}
Combining  \eqref{S(j,l,M) second version} and \eqref{discrepancy bound 2}, we find that
\begin{align*}
S_{j,\ell, M} = \frac{2x^{\frac{5}{2}}\ell^{\frac{3}{2}}}{qLM^{\frac{5}{2}}} + O\left( \frac{\ell^{\frac{3}{2}} x^{\frac{5}{2}}}{qK M^{\frac{5}{2}}} +
\frac{\ell^{\frac{5}{4}}x^{\frac{9}{4}}K^{\frac{1}{2}}}{M^{\frac{9}{4}}}+\frac{\ell x^2K}{M^2}+\frac{\ell^{\frac{7}{4}}x^{\frac{7}{4}}}{M^{\frac{7}{4}}K^{\frac{1}{2}}}+\frac{\ell^{\frac{3}{2}}x^{\frac{7}{4}}}{qM^{\frac{5}{2}}}+\frac{\ell^{\frac{1}{2}} x^{\frac{5}{2}}}{qLM^{\frac{5}{2}}}\right).
\end{align*}
We now substitute the above expression into \eqref{S j} to obtain
\begin{align}
    S_j = \frac{4x^{\frac{5}{2}}}{5qL} + O\left( \frac{ x^{\frac{5}{2}}}{qK } +
x^{\frac{9}{4}}K^{\frac{1}{2}}+x^2K+\frac{Mx^{\frac{7}{4}}}{K^{\frac{1}{2}}}+\frac{x^{\frac{7}{4}}}{q}+\frac{x^{\frac{5}{2}}}{qLM}\right).\label{S q version 2}
\end{align}
Finally, combining \eqref{A(b,q,x) expression} and \eqref{S q version 2},  and applying Euler-Maclaurin summation to the sum over $j$, we have
\begin{align*}
    M(b,q,x) &= \bigg( \frac{2}{\sqrt{3}} \bigg)\sum_{j=1}^{L/2} \bigg(\frac{j}{L}\bigg)\bigg(\frac{4x^{\frac{5}{2}}}{5qL} + O\bigg( \frac{ x^{\frac{5}{2}}}{qK } +
x^{\frac{9}{4}}K^{\frac{1}{2}}+x^2K+\frac{Mx^{\frac{7}{4}}}{K^{\frac{1}{2}}}+\frac{x^{\frac{7}{4}}}{q}+\frac{x^{\frac{5}{2}}}{qLM}\bigg)\bigg)+ O\bigl( x^{\frac{3}{2}} \bigr)\\
&=\frac{2}{\sqrt{3}L} \bigg(\frac{L^2}{8}+\frac{L}{4}+O(1)\bigg)\bigg(\frac{4x^{\frac{5}{2}}}{5qL} + O\bigg( \frac{ x^{\frac{5}{2}}}{qK } +
x^{\frac{9}{4}}K^{\frac{1}{2}}+x^2K+\frac{Mx^{\frac{7}{4}}}{K^{\frac{1}{2}}}+\frac{x^{\frac{7}{4}}}{q}+\frac{x^{\frac{5}{2}}}{qLM}\bigg)\bigg)\\
&=\bigg(\frac{L}{4\sqrt{3}}+O(1)\bigg)\bigg(\frac{4x^{\frac{5}{2}}}{5qL} + O\bigg( \frac{ x^{\frac{5}{2}}}{qK } +
x^{\frac{9}{4}}K^{\frac{1}{2}}+x^2K+\frac{Mx^{\frac{7}{4}}}{K^{\frac{1}{2}}}+\frac{x^{\frac{7}{4}}}{q}+\frac{x^{\frac{5}{2}}}{qLM}\bigg)\bigg)\\
&=\frac{x^{\frac{5}{2}}}{5q\sqrt{3}}+O\left(\frac{x^{\frac{5}{2}}}{qL}+\frac{Lx^{\frac{5}{2}}}{qK}+Lx^{\frac{9}{4}}K^{\frac{1}{2}}+Lx^2K+\frac{MLx^{\frac{7}{4}}}{K^{\frac{1}{2}}}+\frac{Lx^{\frac{7}{4}}}{q}+\frac{x^{\frac{5}{2}}}{qM}\right).
\end{align*}
Taking $M=\frac{K^{1/4}x^{3/8}}{q^{1/2}L^{1/2}}$, we have
\begin{align*}
    M(b,q,x)=\frac{x^{5/2}}{5q\sqrt{3}}+O\left(\frac{x^{5/2}}{qL}+\frac{Lx^{5/2}}{qK}+Lx^{9/4}K^{1/2}+Lx^2K+\frac{Lx^{7/4}}{q}+\frac{L^{1/2}x^{17/8}}{K^{1/4}q^{1/2}}\right).
\end{align*}
To optimize the above result, we let
\[
L= \min\left\{K^{1/2},\frac{x^{1/8}}{q^{1/2}K^{1/4}}, \frac{x^{1/4}}{q^{1/2}K^{1/2}},x^{3/8},\frac{K^{1/6}x^{1/4}}{q^{1/3}} \right\}.
\]
We then have
\begin{align*}
    M(b,q,x)
    &=\frac{x^{5/2}}{5q\sqrt{3}}+O\left(\frac{x^{5/2}}{qK^{1/2}}+\frac{x^{19/8}K^{1/4}}{q^{1/2}}+\frac{x^{9/4}K^{1/2}}{q^{1/2}}+\frac{x^{17/8}}{q}+\frac{x^{9/4}}{q^{2/3}K^{1/6}}\right).
\end{align*}
Taking $K= \frac{x^{1/6}}{q^{2/3}}$, we finally obtain
\begin{align*}
    A(b,q,x)=\frac{1}{x}M(b,q,x)=\frac{x^{3/2}}{5q\sqrt{3}}+O\left(\frac{x^{17/12}}{q^{2/3}}\right).
\end{align*}
This completes the proof of Theorem \ref{thm: A(x) and M(x) in AP}.
\end{proof}
With Theorem \ref{thm: A(x) and M(x) in AP}, we can now prove Corollary \ref{thm: asymptotic for dirichlet series of chi(n)a(n)}.

\begin{proof}[Proof of Corollary \ref{thm: asymptotic for dirichlet series of chi(n)a(n)}]

Since $\chi$ is periodic with period $q$, we have
\begin{align*}
        \sum_{1\leq n\leq x}a_n\chi(n) &= \sum_{b\bmod q} \chi(b)\sum_{\substack{1\leq n\leq x\\n\equiv b\bmod q}}a_n\\
        &=\sum_{b\bmod q} \chi(b)\left(\frac{1}{5q\sqrt{3}}x^{5/2}+O\left(\frac{x^{29/12}}{q^{2/3}}\right)\right),
    \end{align*}
where the last equality follows from Theorem \ref{thm: A(x) and M(x) in AP}. Since
\[
\sum_{b\bmod q} \chi(b) =
\begin{cases}
\varphi(q), & \text{if } \chi = \chi_0, \\
0, & \text{if } \chi \neq \chi_0,
\end{cases}
\]
we have 
\begin{align}\label{sum of an*chi(n)}
\sum_{1\leq n\leq x}a_n\chi(n)  =
\begin{cases}
\frac{\varphi(q)}{5q\sqrt{3}}x^{5/2}+O\left(\frac{\varphi(q)x^{29/12}}{q^{2/3}}\right), & \text{if } \chi = \chi_0, \\
O\left(\frac{\varphi(q)x^{29/12}}{q^{2/3}}\right), & \text{if } \chi \neq \chi_0.
\end{cases}
\end{align}
This completes the proof of Corollary \ref{thm: asymptotic for dirichlet series of chi(n)a(n)}.

\end{proof}

\section{Proof of Theorem \ref{thm: asymptotic for bn}}\label{sec: proof of thm: asymptotic for bn}
\begin{proof}
    In order to estimate 
   \begin{align}
   B(x):= \sum_{1\leq n\leq x}b_n \left(1-\frac{n}{x}\right),
    \end{align}
     we start by associating it with a Dirichlet series and then apply Perron's formula. We begin with the observation that
\begin{align}\label{dirichlet convolution}
    b_{n} = \sum_{d | n} a_{d} = (a \ast 1)(n).
\end{align}
Let $s$ be any complex number, and define $F(s)$ and $B(s)$ as the Dirichlet series associated to the sequences $(a_{n})$ and $(b_n)$ respectively, that is, 
\begin{align}
    F(s) := \sum^{\infty}_{n=1} \frac{a_{n}}{n^{s}},\textrm{\quad and\quad} H(s) := \sum^{\infty}_{n=1} \frac{b_{n}}{n^{s}}.\label{def of F(s) and H(s)}
\end{align}
From \eqref{dirichlet convolution}, we have 
\begin{align}\label{H(s)=F(s)zeta(s)}
    H(s) = F(s) \zeta(s).
\end{align}
Thus, we wish to analyze $F(s)$ further. Using \cite[Equations\ (4.3), (4.4), and (4.6)]{equidistribution_paper_1}, we have 
\begin{align*}
    F(s)\leq \sum_{n=1}^\infty\frac{|a_n|}{n^s}\leq\sum_{n=1}^\infty \frac{2}{\sqrt{3}n^{s-\frac{3}{2}}},
\end{align*}
which converges when $\text{Re}(s)>5/2$. Since $\zeta(s)$ converges absolutely when $\text{Re}(s)>1$, we conclude that $H(s)$ converges absolutely when $\text{Re}(s)>5/2$. 

By the asymptotic formula for $A(x)$ provided by \cite[Thm\ 1.1]{equidistribution_paper_1}, we have
\begin{align}\label{approximation of (an)}
    \sum_{1\leq n\leq x}a_n -\sum_{1\leq n\leq x} \frac{1}{2\sqrt{3}}n^{3/2}\ll x^{29/12}. 
\end{align}
Defining $G(s)$ as
\begin{align}
    G(s) := \sum^{\infty}_{n=1} \frac{c_{n}}{n^{s}},
\end{align}
where $c_{n} = a_{n} - \frac{1}{2\sqrt{3}} n^{3/2}$, we then have 
\begin{align}
    F(s)=G(s)+\frac{1}{2\sqrt{3}}\zeta\left(s-\frac{3}{2}\right).\label{relation btw F(s) and G(s)}
\end{align}
With \eqref{approximation of (an)}, we have 
\begin{align}\label{sum of cn bound}
\sum_{1\leq n\leq x} c_n\ll X^{29/12}. 
\end{align}
Now for $s\in\C$ such that Re$(s)>29/12$, we apply partial summation to obtain
\begin{align*}
    \lim_{x\rightarrow\infty}\bigg|\sum^{x}_{n=1}\frac{c_n}{n^s}\bigg|&=\lim_{x\rightarrow\infty}\bigg|\bigg(\sum_{n=1}^xc_n\bigg)\frac{1}{x^{s}}-\int_{1}^x\bigg(\sum_{n=1}^tc_n\bigg)\bigg(-st^{-s-1}\bigg)\dd t\bigg|\\
    &\leq \lim_{x\rightarrow\infty}\bigg(\bigg|\bigg(\sum_{n=1}^xc_n\bigg)\frac{1}{x^{s}}\bigg|+\int_{1}^x\bigg|\bigg(\sum_{n=1}^tc_n\bigg)\bigg(-st^{-s-1}\bigg)\bigg|\dd t\bigg)\\
    &\ll\lim_{x\rightarrow\infty} x^{29/12-|s|}+|s|\int_1^x t^{17/12-|s|}\dd t\\
    &\ll \lim_{x\rightarrow\infty}x^{29/12-|s|},
\end{align*}
which converges when $\text{Re}(s) > 29/12$. Therefore, 
\[
G(s) = \sum_{n=1}^\infty \frac{c_n}{n^s}
\]
converges when $\text{Re}(s) > 29/12$, which is less than $5/2$. Together with the fact that both $G(s)$ and $\zeta(s-3/2)$ converge absolutely when $\text{Re}(s)> 5/2$, we conclude that $F(s)$ has a simple pole at $s=5/2$ with residue $\frac{1}{2\sqrt{3}}$. Therefore, using \eqref{H(s)=F(s)zeta(s)}, we conclude that $H(s)$ has analytic continuation to $\text{Re}(s)>29/12$ except a simple pole at $s=5/2$ with residue $\frac{\zeta(5/2)}{2\sqrt{3}}$. Next, we take a parameter $T\geq 1$, which will be optimized later, and fix a $\sigma_0 >5/2$. Upon applying Perron's Formula, we obtain 
\begin{align*}
 \sum_{1\leq n\leq x} b_n\left(1-\frac{n}{x}\right)& = \frac{1}{2\pi i}\int_{\sigma_0-i\infty}^{\sigma_0+i\infty}H(s)\frac{x^s}{s(s+1)}\dd s\\
 &=\frac{1}{2\pi i}\int_{\sigma_0-iT}^{\sigma_0+iT}H(s)\frac{x^s}{s(s+1)}\dd s+R(x,\sigma,T),
\end{align*}
where $R(x,\sigma,T)\ll x^{\sigma_0}/T$. Now choose a sufficiently small $\delta>0$, and let $\mathcal{C}$ be the incomplete contour connecting the points $\sigma_0+iT$, $\frac{29}{12}+\delta+iT$, $\frac{29}{12}+\delta-iT$, and $\sigma_0-iT$. Then by the residue theorem, we get
\begin{align*}
    \frac{1}{2\pi i}\int_{\sigma_0-iT}^{\sigma_0+iT}H(s)\frac{x^s}{s(s+1)}\dd s &= -\frac{1}{2\pi i}\int_{\mathcal{C}}H(s)\frac{x^s}{s(s+1)}\dd s+\text{Res}\left(\frac{H(s)x^s}{s(s+1)},s=\frac{5}{2}\right)\\
    &=-\frac{1}{2\pi i}\int_{\mathcal{C}}H(s)\frac{x^s}{s(s+1)}\dd s+\frac{2\zeta(5/2)x^{5/2}}{35\sqrt{3}}.
\end{align*}
Now let $s = \sigma+it$. Then 
\begin{align}
   \int^{\sigma_0+iT}_{\frac{29}{12}+\delta+iT}H(s)\frac{x^s}{s(s+1)}\dd s
   &= \int^{\sigma_0}_{\frac{29}{12}+\delta}F(\sigma+iT) \zeta(\sigma+iT)\frac{x^{\sigma+iT}}{(\sigma+iT)(\sigma+1+iT)}\dd \sigma\notag\\
    &=\int^{\sigma_0}_{\frac{29}{12}+\delta}G(\sigma+iT)\zeta(\sigma+iT)\frac{x^{\sigma+iT}}{(\sigma+iT)(\sigma+1+iT)}\dd \sigma\notag\\
   &\quad+\frac{1}{2\sqrt{3}}\int^{\sigma_0}_{\frac{29}{12}+\delta}\zeta(\sigma-\frac{3}{2}+iT)\zeta(\sigma+iT)\frac{x^{\sigma+iT}}{(\sigma+iT)(\sigma+1+iT)}\dd \sigma.\label{sum of two integrals}
\end{align}
To evaluate the first integral on the right-hand side of \eqref{sum of two integrals}, recall that $G(s)$ converges when $s>7/3$. Thus, applying Mellin transform, with $\sigma\in (29/12+\delta, \sigma_0)$,
\begin{align}\label{mellin transform applied}
    G(\sigma+iT) = (\sigma+iT)\int_{1}^\infty \frac{\sum_{n=1}^xc_n}{x^{\sigma+1+iT}}dx\ll T,
\end{align}
so the first integral on the far right side of \eqref{sum of two integrals} is 
\[
\ll T\int_{\frac{29}{12}+\delta}^{\sigma_0}\left|\frac{x^{\sigma+iT}}{(\sigma+iT)(\sigma+1+iT)}\zeta(\sigma+iT)\right|\dd \sigma\ll T\int_{\frac{29}{12}+\delta}^{\sigma_0}\frac{x^{\sigma}}{T^2} \dd \sigma\ll \frac{x^{\sigma_0}}{T\log x}.
\]
Similarly, by known bounds for the Riemann zeta function in the critical strip (for example, see Equa. (5.20) in \cite{IwaniecKowalski}), the second integral on the far right-hand side of \eqref{sum of two integrals} is
\[
\ll \int_{\frac{29}{12}+\delta}^{\sigma_0}\left|\zeta(\sigma-\frac{3}{2}+iT)\zeta(\sigma+iT)\frac{x^{\sigma}}{T^2}\right|\dd \sigma\ll \frac{x^{\sigma_0}}{T\log x}.
\]
Similar integral estimates can be done for the remaining sub-paths of the contour $\mathcal{C}$. By symmetry, we have 
\begin{align*}
   \int_{\frac{29}{12}+\delta-iT}^{\sigma_0-iT}H(s)\frac{x^s}{s(s+1)}\dd s
   &= \int_{\frac{29}{12}+\delta}^{\sigma_0}F(\sigma-iT) \zeta(\sigma-iT)\frac{x^{\sigma-iT}}{(\sigma-iT)(\sigma+1-iT)}\dd \sigma\notag\\
   &\ll \frac{x^{\sigma_0}}{T\log x}.
\end{align*}
Next, by \eqref{relation btw F(s) and G(s)} and \eqref{mellin transform applied} with $T$ replaced by $t$, as well as by known bounds for the Riemann zeta function in the critical strip, we have
\[
\left|F\left(\frac{29}{12}+\delta+it\right)\right|\leq \left|G\left(\frac{29}{12}+\delta+it\right)\right|+\left|\frac{1}{2\sqrt{3}}\zeta\left(\frac{11}{12}+\delta+it\right)\right|\ll |t|+1. 
\]
Thus,
\begin{align*}
    \int_{\frac{29}{12}+\delta-iT}^{\frac{29}{12}+\delta+iT}H(s)\frac{x^s}{s(s+1)}\dd s &= \int_{-T}^{T}F\left(\frac{29}{12}+\delta+it\right) \zeta\left(\frac{29}{12}+\delta+it\right)\frac{x^{\frac{29}{12}+\delta+it}}{(\frac{29}{12}+\delta+it)(\frac{41}{12}+\delta+it)}dt\\
    &\ll_\delta x^{\frac{29}{12}+\delta}\int_{-T}^{T} \frac{1}{|t|+1}dt\ll_\delta x^{\frac{29}{12}+\delta}\log T.
\end{align*}
Combining all error terms, we have
\begin{align*}
    \sum_{1\leq n\leq x} b_n\left(1-\frac{n}{x}\right)& =\frac{2\zeta(5/2)x^{5/2}}{35\sqrt{3}}+E(x,\sigma,T),
\end{align*}
where 
\begin{align*}
    E(x,\sigma,T) \ll_\delta \frac{x^{\sigma_0}}{T}+x^{\frac{29}{12}+\delta}\log T.
\end{align*}
Take $T=x^{\sigma_0-29/12-\delta}$,  we then obtain 
\[
\frac{x^{\sigma_0}}{T}\ll_\delta x^{29/12+\delta}, \quad x^{\frac{29}{12}+\delta}\log T\ll_\delta x^{29/12+\delta},
\]
for any $\delta>0$. Therefore, $E(x,\sigma,T)\ll x^{29/12+\delta}$ for any $\delta>0$, and
Theorem \ref{thm: asymptotic for bn} follows.
\end{proof}
\begin{rem}
   The asymptotic formula for the Dirichlet series associated with $\sum_{1\leq n\leq x}a_n\chi(n)$ defined in Corollary \ref{thm: asymptotic for dirichlet series of chi(n)a(n)} can be obtained as follows. Since $|a_n\chi(n)|=|a_n|$, and by Theorem \ref{thm: asymptotic for bn}, $F(s)$, defined in \eqref{def of F(s) and H(s)}, converges absolutely when $\text{Re}(s)>5/2$, so does $F_\chi(s)$. 
Therefore, by the Mellin transform, we arrive at
\begin{align*}
  F_\chi(s) =  s\int_1^\infty \sum_{1\leq n\leq x}a_n\chi(n) x^{-s-1}dx,
\end{align*}
for any $s>5/2$. Substituting \eqref{sum of an*chi(n)} in the above expression, if $\chi=\chi_0$, 
\begin{align*}
    F_\chi(s) &= s\int_1^\infty \left(\frac{\varphi(q)}{5q\sqrt{3}}x^{5/2}+O\left(\frac{\varphi(q)x^{29/12}}{q^{2/3}}\right)\right) x^{-s-1}dx\\
    &=\frac{s\varphi(q)}{5q\sqrt{3}(s-5/2)}+O\left( \frac{s\varphi(q)}{(s - \frac{29}{12})q^{2/3}}  \right).
\end{align*}
Otherwise, when $\chi$ is not the principal character, 
\begin{align*}
    F_\chi(s) &=O\left( \frac{s\varphi(q)}{(s - \frac{29}{12})q^{2/3}}  \right).
\end{align*}
\end{rem}

\end{document}